\documentclass[pdflatex,sn-mathphys-num]{sn-jnl}

\usepackage{graphicx}%
\usepackage{multirow}%
\usepackage{amsmath,amssymb,amsfonts}%
\usepackage{amsthm}%
\usepackage{mathrsfs}%
\usepackage[title]{appendix}%
\usepackage{xcolor}%
\usepackage{textcomp}%
\usepackage{manyfoot}%
\usepackage{booktabs}%
\usepackage{algorithm}%
\usepackage{algorithmicx}%
\usepackage{algpseudocode}%
\usepackage{listings}%
\usepackage[utf8]{inputenc}
\usepackage{doi}

\theoremstyle{thmstyleone}%
\newtheorem{theorem}{Theorem}
\newtheorem{corollary}[theorem]{Corollary}%

\theoremstyle{thmstyletwo}%

\theoremstyle{thmstylethree}%
\newtheorem{definition}{Definition}%

\begin{document}

\title[]{Cryptoanalysis of a tropical triad matrix semiring key exchange protocol}


\author*[1]{\fnm{Alvaro} \sur{Otero Sanchez}}\email{aos@ual.es}

\affil*[1]{\orgdiv{Department}, \orgname{Universidad de Almería}, \orgaddress{\street{Carretera de Sacramento, s/n}, \city{La Cañada de San Urbano}, \postcode{04120}, \state{Almería}, \country{España}}}


\abstract{This article analyzes a key exchange protocol based on the triad tropical semiring, recently proposed by Jackson, J. and Perumal, R. We demonstrate that the triad tropical semiring is isomorphic to a circulant matrix over tropical numbers. Consequently, matrices in this semiring can be represented as tropical matrices. As a result, we conduct a cryptanalysis of the key exchange protocol using an algorithm introduced by Sulaiman Alhussaini, Craig Collett, and Serge˘ı Sergeev to solve the double discrete logarithm problem over tropical matrices.}




\maketitle

\section{Introduction}\label{sec1}

Modern cryptography relies on the so called discrete logarithm problem over cyclic group, and in particular over elliptic curves. However, in 1994 Shor presented a quantum algorithm able to solve this problem in polynomial time. This was the beginning of the post quatum cryptography; the research of new cryptography protocols that remains safe under the assumption of the existence of a quantum computer. 

Maze et al. \cite{maze2007} introduced a general framework for defining key-exchange protocols based on a semigroup action on a set. Their work can be seen as a generalization of the Diffie-Hellman \cite{diffie1976new} and ElGamal \cite{elgamal1985public} schemes within an algebraic setting. In their original paper, they proposed an example using a finite simple semiring, which was recently subjected to cryptanalysis in \cite{otero2024cryptanalysis}. However, several cryptographic protocols have since been developed following the ideas of Maze et al. For instance, in \cite{kahrobaei2016group}, Kahrobaei and Koupparis explored a key-exchange protocol based on non-commutative group actions, extending the original framework. Other examples include the work of Gnilke and Zumbrägel \cite{gnilke2024cryptographic}, who linked Maze et al.'s ideas to recent advances in isogeny-based cryptography. Additionally, \cite{Olvera19} introduced a key-exchange protocol based on twisted group rings, incorporating a group key-exchange agreement.

In this research, Grigoriev and Shpilrain proposed the use of tropical semiring for public key exchange \cite{Grigoriev2013}, \cite{Grigoriev2019} and for digital signatures \cite{Chen2024}. Nevertheless, the first attempt was analyced by \cite{Kotov2018}, where it was introduced the so called Kotov-Ushakov attack, an heuristic attack that knowadays has become the standar attack against tropical cryptography. In \cite{OteroSanchez2024}, the authors propose a new deterministic attack against a public key exchange protocol based on tropical semiring, that improve the Kotov-Ushakov attack in the sense that it can be used in the sames scenarios, but with a deterministic output. In addition, the other two tropical cryptographyc protocols proposed by Grigoriev and Shpilrain has been proven not secure in \cite{Muanalifah2022}, \cite{Isaac2021} and \cite{Rudy2020}. 

Far from being abandoned, this ideas has been further explored. In \cite{Jackson2024}, the authors recently proposed the use of triad matrix semiring. This semiring is a generalization of tropical semiring where elements are a vector of 3 entries, and with a modified addition and multiplication that endow them with the structure of semiring. Based on them, the authors stablish a public key exchange protocol that use matrix with entries over triad semiring.

\section{Results}\label{sec2}

In this paper, we introduce an isomorphism between triad semiring and circuland matrix over tropical semiring. As a result, it is possible to satblish an isomorhism between triad matrix semiring and tropical matrix semiring, and therefore reinterpretate the public key exchange in terms of tropical matrix. This prove that the previous protocol can be seen as the Stickel protocol where instead of taking a polynomial, only a monomial is used. Finally, this protol is suscentible to the attack introduced in \cite{cryptoeprint:2024/010} .

\section{Preliminaries}\label{sec3}
In this secition we will introduce some basic background on tropical semiring as well as triad semiring. 

\begin{definition}
    A semiring $R$ is a non-empty set together with two operations $+$ and $\cdot$ such that $(S,+)$ is a commutative monoid, $(S,\cdot)$ is a monoid and 
    the following distributive laws hold:

    \begin{equation*}
        \begin{split}
            a(b+c) & = ab + ac \\
            (a+b)c & = ac + bc
        \end{split}
    \end{equation*}
    
    \noindent We say that $(R,+,\cdot)$ is additively idempotent if $a+a=a$ for all $a \in R$.
\end{definition}

\begin{definition}
    Let $R$ be a semiring and $(M,+)$ be a commutative semigroup with identity $0_M$. $M$ is a right semimodule over $R$ if there is an external operation  $\cdot \ : M\times R\rightarrow M$ such that

    \begin{equation*}
    \begin{split}
             (m\cdot a)\cdot b & = m\cdot (a\cdot b),\\
             m\cdot (a+b) & = m\cdot a + m\cdot b, \\
             (m+n)\cdot a & = m\cdot a +n\cdot a, \\
             0_M\cdot a & = 0_M,
    \end{split}
    \end{equation*}

    \noindent for all $a, b \in R$ and $m,n \in M$. We will denote  $m\cdot a$ by simple concatenation $ma$. 
\end{definition}

\begin{definition}
  Let $\mathbb{T}=\mathbb{Z} \cup \{ -\infty \}$. The tropical semiring is the semiring $(\overline{\mathbb{Z}},\oplus,\odot)$, where 
  \begin{align*}
     a \oplus b &= \max\{a,b\} \\
     a \odot b &= a+b
  \end{align*}
\end{definition}
Using tropical semiring, in \cite{Jackson2024} it is introduced the trial semiring
\begin{definition}
    Let $\overline{\mathbb{T}}=\mathbb{T}\times\mathbb{T}\times\mathbb{T}$. Trial semiring is the semiring $\left(\overline{\mathbb{T}},\overline{\oplus},\overline{\odot}\right)$ where
    \begin{align*}
        (a, b, c) \oplus (d, e, f) & = (a \oplus d, b \oplus e, c \oplus f) \\
        (a, b, c) \odot (d, e, f) & = \left( (a \odot d) \oplus (b \odot f) \oplus (c \odot e), (a \odot e) \oplus (b \odot d) \oplus (c \odot f), (a \odot f) \oplus (b \odot e) \oplus (c \odot d) \right)
    \end{align*}
    for all $a,b,c,d,e,f \in \mathbb{T}$
\end{definition}
In \cite{Jackson2024} it is proven that the previous definition is a semiring
\begin{theorem}
    $\left(\overline{\mathbb{T}},\overline{\oplus},\overline{\odot}\right)$ is a semiring
\end{theorem}
The set of square matrix of $n\times n$ over trial semiring will be denoted by $\mathbb{M}_n(\overline{\mathbb{T}})$.

Finally, we will introduce circulant matrix. 
\begin{definition}
Let \( C \in \mathbb{M}_n(S) \) with $S$ a semiring. The matrix \( C \) is called a circulant matrix if $C$ of the form 
\[
C = \begin{pmatrix}
c_0 & c_1 & c_2 & \cdots & c_{n-1} \\
c_{n-1} & c_0 & c_1 & \cdots & c_{n-2} \\
c_{n-2} & c_{n-1} & c_0 & \cdots & c_{n-3} \\
\vdots & \vdots & \vdots & \ddots & \vdots \\
c_1 & c_2 & c_3 & \cdots & c_0
\end{pmatrix}.
\]
with $c_i \in S \forall i = 0,\cdots, n-1$. In this case, $C$ will be denoted as $C[c_0, c_1, \dots, c_{n-1}]$. The set of all circulant matrix of $\mathbb{M}_n(S)$ will be denoted as $Circ\mathbb{M}_n(S)$
\end{definition}

\section{Double action over tropical}
In \cite{cryptoeprint:2024/010}, the authors explore the tropical two-sided discrete logarithm 
\begin{definition}
   Given $D_1, D_2, M, U \in \mathbb{M}_n(\mathbb{T})$ such that $U=D_1^{\odot t_1}M D_2^{\odot t_2}$ for some $t_1,t_2 \in \mathbb{N}$. The tropical two-sided discrete logarithm  is to find $t'_1,t'_2 \in \mathbb{N}$ such that $U=D_1^{\odot t'_1}M D_2^{\odot t'_2}$
\end{definition}
To solve this problem, they introduce the maximum cycle mean of a tropical matrix
\begin{definition}
    For \( A \in \mathbb{M}_n(\mathbb{T} \), the maximum cycle mean \( \lambda(A) \) is defined as

\[
\lambda(A) = \bigoplus_k \bigoplus_{i_1, \dots, i_k \in [n]} \sqrt[\otimes k]{a_{i_1 i_2} \otimes \dots \otimes a_{i_k i_1}.}
\]
\end{definition}
If we consider $A\in \mathbb{M}_n(\mathbb{T}$ as the matrix of weith of a graph of $n$ vertex, we can then compute it critical cycles. We also need the kleen star
\begin{definition}
    Let $A\in \mathbb{M}_n(\mathbb{T})$ such that $\lambda(A)\leq 0$, then the kleen star $A^*$ is the following matrix
    \[
    A^* = I \oplus A \oplus A^{\odot 2} \oplus \cdots A^{\odot n-1}
    \]
    with $I$ the tropical identity in $\mathbb{M}_n(\mathbb{T})$
\end{definition}
The main theorem of \cite{cryptoeprint:2024/010} is the following one.
\begin{theorem}
Let \( A \in \mathbb{R}^{n \times n}_{\max} \) have \( \lambda = \lambda(A) \neq -\infty \), and let \( Z \) be a critical cycle of \( A \) with length \( l_Z \). Then for some integer \( T_{weak} \), we have:

\[
A^{\odot t} = \lambda^{\odot t} \odot C_Z \odot S^{\odot t}_Z \odot R_Z \oplus B^{\odot t}_Z \quad \forall t \geq \text{Tweak},
\]

and

\[
A^{\odot t} = \lambda^{\odot t} \odot \left( C_Z \odot S^{\odot t (\text{rem} \, l_Z)}_Z \odot R_Z \right) \oplus B^{\odot t}_Z \quad \forall t \geq \text{Tweak},
\]

where \( t (\text{rem} \, l_Z) \) is the remainder when \( t \) is divided by \( l_Z \), and \( C_Z \), \( S_Z \), \( R_Z \), and \( B_Z \) are defined by:

\[
(C_Z)_{ij} =
\begin{cases}
(U_Z)_{ij} & \text{if } j \in Z, \\
-\infty & \text{otherwise},
\end{cases}
\]

\[
(R_Z)_{ij} =
\begin{cases}
(U_Z)_{ij} & \text{if } i \in Z, \\
-\infty & \text{otherwise},
\end{cases}
\]

\[
(S_Z)_{ij} =
\begin{cases}
(a_{ij} \odot \lambda^{-1}) & \text{if } (i, j) \in Z, \\
-\infty & \text{otherwise},
\end{cases}
\]

\[
(B_Z)_{ij} =
\begin{cases}
-\infty & \text{if } i \in Z \text{ or } j \in Z, \\
a_{ij} & \text{otherwise},
\end{cases}
\]

where \( U_Z = \left( (A \odot \lambda^{-1})^{\odot l_Z} \right)^* \) (i.e., the Kleene star of \( (A \odot \lambda^{-1})^{\odot l_Z} \)).
\end{theorem}
This theorem is use to stablish the following equation
\begin{equation} \label{Eq5}
\left\{
\begin{aligned}
    \beta - \tau - \lambda_1 \cdot \bar{t}_1 - \lambda_2 \cdot \bar{t}_2 &= (\lambda_1 \cdot l_Z) \cdot x + (\lambda_2 \cdot l_W) \cdot y \\
    x &\geq \frac{(n-1) l_Z - \bar{t}_1}{l_Z} \\
    y &\geq \frac{(n-1) l_W - \bar{t}_2}{l_W}
\end{aligned}
\right.
\end{equation}
Finally, they present the following algorithm

\begin{algorithm} \label{Algorithm}
\caption{Solving the tropical two-sided discrete logarithm with shift using CSR}
\label{alg:TTDL}
\begin{algorithmic}[1]
\Require \( U, D_1, M, D_2 \)
\Ensure \( t'_1, t'_2, \tau \)
\State Calculate \( \lambda(D_1) = \lambda_1 \), \( \lambda(D_2) = \lambda_2 \)
\State Find a critical cycle \( Z \) from \( D_1 \) and \( W \) from \( D_2 \), with lengths \( l_Z \) and \( l_W \), respectively
\State Compute \( S_Z, R_Z, C_W, S_W \) as in Theorem 2.3
\For{\( \bar{t}_1 = 0 \) to \( l_Z \)}
    \For{\( \bar{t}_2 = 0 \) to \( l_W \)}
        \If{\( (U - (S_Z^{\odot \bar{t}_1} \odot R_Z \odot M \odot C_W \odot S_W^{\odot \bar{t}_2}))_{ij} = \beta \) for some \( \beta \in \mathbb{R} \) and for all \( i, j \) where \( i \in Z \) and \( j \in W \)}
            \State Check if (\ref{Eq5}) is solvable
            \If{(\ref{Eq5}) is solvable}
                \State Return \( (t'_1, t'_2, \tau) \) where:
                \[
                t'_1 = l_Z \cdot x + \bar{t}_1, \quad t'_2 = l_W \cdot y + \bar{t}_2
                \]
            \EndIf
        \EndIf
    \EndFor
\EndFor
\end{algorithmic}
\end{algorithm}

\section{Mathematical Results}\label{sec7}
First, we will start with an isomorphism between trial semiring and tropical circulant matrix. 
\begin{theorem}
    There is an isomorhism between trial semring $\overline{\mathbb{T}}$ and $Circ\mathbb{M}_3(\mathbb{T})$ given by
    \begin{align*}
    \psi : \overline{\mathbb{T}} & \to \text{Circ}\, \mathbb{M}_3(\mathbb{T}) \\
    (a,b,c) & \mapsto C[a,b,c]
\end{align*}
\end{theorem}
\begin{proof}
    Clearly $\psi$ is biyective, and we only have to prove it is a ring homomorphism. 

\underline{Addition}\\
\begin{align*}
    \psi((a,b,c)\overline{\oplus}(d,e,f))& =\psi(a\oplus d,b \oplus e, c \oplus f)) \\ &= C[a\oplus d,b \oplus e, c \oplus f] \\ &= C[a,b,c] \oplus C[d,e,f] \\ & = \psi(a,b,c) \overline{\oplus} \psi(d,e,f)
\end{align*}
\underline{Multiplication}\\
\begin{align*}
    \psi((a,b,c)\overline{\odot}(d,e,f))& =\psi\left( (a \odot d) \oplus (b \odot f) \oplus (c \odot e), (a \odot e) \oplus (b \odot d) \oplus (c \odot f), (a \odot f) \oplus (b \odot e) \oplus (c \odot d) \right) \\ &= C[\left( (a \odot d) \oplus (b \odot f) \oplus (c \odot e), (a \odot e) \oplus (b \odot d) \oplus (c \odot f), (a \odot f) \oplus (b \odot e) \oplus (c \odot d) \right)] \\ &= \begin{pmatrix}
(a \odot d) \oplus (b \odot f) \oplus (c \odot e) & (a \odot f) \oplus (b \odot e) \oplus (c \odot d) &(a \odot e) \oplus (b \odot d) \oplus (c \odot f)  \\
(a \odot e) \oplus (b \odot d) \oplus (c \odot f) & (a \odot d) \oplus (b \odot f) \oplus (c \odot e) & (a \odot f) \oplus (b \odot e) \oplus (c \odot d) \\
(a \odot f) \oplus (b \odot e) \oplus (c \odot d) & (a \odot e) \oplus (b \odot d) \oplus (c \odot f) & (a \odot d) \oplus (b \odot f) \oplus (c \odot e) \\
\end{pmatrix}.
\\ &= \begin{pmatrix}
a & c & b  \\
b & a & c \\
c & b & a
\end{pmatrix}
\begin{pmatrix}
d & f & e  \\
e & d & f \\
f & e & d 
\end{pmatrix}
\\ & = \psi((a,b,c)) \odot \psi((d,e,f))
\end{align*}
\end{proof}
As a result, we have the following colorary
\begin{corollary} \label{IsomorphismMatrix}
    There exist an isomorhism $\mathbb{M}_n(\overline{\mathbb{T}}) \cong \mathbb{M}_n(Circ\mathbb{M}_3(\mathbb{T})) \cong J  \leq_{Ring} \mathbb{M}_{3n}(\mathbb{T})
$ with $J$ subring of $\mathbb{M}_{3n}(\mathbb{T})$
\end{corollary}
\begin{proof}
    The first isomorhism is clear. The second came from the ismorphism $\mathbb{M}_n(\mathbb{M}_m(S)) \cong \mathbb{M}_{nm}(S)$ for all ring $S$, given by $\Psi :\mathbb{M}_n(\mathbb{M}_m(S)) \longrightarrow \mathbb{M}_{nm}(S)$, with
    \begin{align*}
    \Psi :\mathbb{M}_n(\mathbb{M}_m(S)) & \longrightarrow \mathbb{M}_{nm}(S) \\
\begin{pmatrix}
A_{11} & A_{12} & \cdots & A_{1n} \\
A_{21} & A_{22} & \cdots & A_{2n} \\
\vdots & \vdots & \ddots & \vdots \\
A_{n1} & A_{n2} & \cdots & A_{nn}
\end{pmatrix} & \mapsto 
\left(
\begin{array}{c|c|c|c}
\begin{matrix} A_{11} \end{matrix} & \begin{matrix} A_{12} \end{matrix} & \cdots & \begin{matrix} A_{1n} \end{matrix} \\
\hline
\begin{matrix} A_{21} \end{matrix} & \begin{matrix} A_{22} \end{matrix} & \cdots & \begin{matrix} A_{2n} \end{matrix} \\
\hline
\vdots & \vdots & \ddots & \vdots \\
\hline
\begin{matrix} A_{n1} \end{matrix} & \begin{matrix} A_{n2} \end{matrix} & \cdots & \begin{matrix} A_{nn} \end{matrix} \\
\end{array}
\right)
    \end{align*}
\end{proof}

\section{Key exchange}
In \cite{Jackson2025}, the following key exchange protocol is proposed

Suppose that two IIoT devices, \( D_1 \) and \( D_2 \), want to share some data. First, they agree on two public matrices \( X, Y \in \mathbb{M}_n(\mathbb{T}) \) and generate the secret key as follows:

\begin{itemize}
    \item \( D_1 \) chooses two natural numbers \( a, b \) and computes:
    \[
    A = X^{\odot a} \odot Y^{\odot b}
    \]
    Then, \( D_1 \) sends \( A \) to \( D_2 \).
    
    \item \( D_2 \) chooses two natural numbers \( c, d \) and computes:
    \[
    B = X^{\odot c} \odot Y^{\odot d}
    \]
    Then, \( D_2 \) sends \( B \) to \( D_1 \).
    
    \item \( D_1 \) computes the secret key:
    \[
    K_A = X^{\odot a} \odot B \odot Y^{\odot b}
    \]
    
    \item \( D_2 \) computes the secret key:
    \[
    K_B = X^{\odot c} \odot A \odot Y^{\odot d}
    \]
\end{itemize}
This public key exchange is based on the assumption that it is hard to solve the following problem.
\begin{definition}
       Given $D_1, D_2, M, U \in \mathbb{M}_n(\mathbb{\overline{T}})$ such that $U=D_1^{\odot t_1}M D_2^{\odot t_2}$ for some $t_1,t_2 \in \mathbb{N}$. The trial two-sided discrete logarithm  is to find $t'_1,t'_2 \in \mathbb{N}$ such that $U=D_1^{\odot t'_1}M D_2^{\odot t'_2}$
\end{definition}
Thanks to the isomorphism given in Corollay \ref{IsomorphismMatrix}, to solve this problem is an instance of the tropical two-side discreth logarithm over $\mathbb{M}_{3n}(\mathbb{T})$. The heuristic algorithm \cite{Algorithm} has an succes rate over closer to $100\%$, and as a result, the proposed protocol is not segure.

\section{Conclusion}\label{sec13}

We have shown arguments against the security of the public key-exchange proposed in \cite{Jackson2025}. We provide a new isomorphism that reduces the two-side discrete logarithm over the triad matrix semiring to the tropical case, where there exist algorithms to solve it in polynomial time. 

\backmatter

\end{document}